%% file: Canonical_parts_ArXiv.tex
\documentclass{article}
\usepackage{caption,subcaption}
\pdfoutput=1 
\include{defs}

\nocomments

\renewcommand{\cE}{{\ensuremath\rm e}}

\newcommand{\partOne}{\cite{CDHH13CanonicalAlg}}

\title{Canonical \td s of finite graphs\\ II. Essential parts}
\author{J.\ Carmesin \and R.\ Diestel \and M.\ Hamann \and F.\ Hundertmark}

\begin{document}

\maketitle

\begin{abstract}\noindent
In Part~I of this series we described three algorithms that construct canonical \td s of graphs which distinguish all their $k$-blocks and tangles of order~$k$. We now establish lower bounds on the number of parts in these decompositions that contain such a block or tangle, and determine conditions under which such parts contain nothing but a $k$-block.
\end{abstract}

\section*{Introduction}
A~\emph{$k$-block} in a graph $G$, where $k$ is any positive integer, is a maximal set $X$ of at least~$k$ vertices such that no two vertices $x,x'\in X$ can be separated in~$G$ by fewer than~$k$ vertices other than $x$ and~$x'$. Thus, $k$-blocks for large~$k$ can be thought of as highly connected pieces of a graph, but their connectivity is measured not in the subgraph they induce but in the ambient graph.

Another concept of highly connected pieces of a graph, formally quite different from $k$-blocks, is the notion of a tangle proposed by Robertson and Seymour~\cite{GMX}. Tangles are not defined directly in terms of vertices and edges, but indirectly by assigning to every low-order separation of the graph one of its two sides, the side in which `the tangle' is assumed to sit. In order for this to make sense, the assignments of sides have to satisfy some consistency constraints, in line with our intuition that one tangle should not be able to sit in disjoint parts of the graph at once.

In a fundamental paper on graph connectivity and tree structure, Hundertmark~\cite{profiles} showed that high-order blocks and tangles have a common generalization, which he called `profiles'. These also work for discrete structures other than graphs. We continue to work with profiles in this paper. All the reader needs to know about profiles is explained in Part~I of this paper~\partOne.

In Part~I we described a family of algorithms which construct, for any finite graph~$G$ and $k\in\N$, a \td\ of $G$ that has two properties: it distinguishes all the $k$-blocks and tangles of order~$k$ in~$G$, so that distinct blocks or tangles come to sit in distinct parts of the decomposition, and it is canonical in that the map assigning this decomposition to $G$ commutes with graph isomorphisms.

~~\llap{In} this follow-up to~\partOne, we study these decompositions in more detail. Given~$k$, let us call a part of such a decomposition \emph{essential\/} if it contains a $k$-block or accommodates a tangle of order~$k$. (Precise definitions will follow.) Since the aim of our \td s is to display how $G$ can be cut up into its highly connected pieces, ideally every part of such a decomposition would be essential, and the essential parts containing a $k$-block would contain nothing else. (This makes no sense for tangles, since they cannot be captured by a set of vertices.)

Neither of these aims can always be attained. Our objective is to see when or to which extent they can. After providing in Section~\ref{OrientingTDs} some background on how \td s relate to oriented \sys s, we devote Section~\ref{sec_kcon} to establishing upper bounds on the number of inessential parts in a canonical \td\ of a graph that distinguishes all its $k$-profiles. These bounds depend in interesting ways on the algorithm chosen to find the decomposition. All the bounds we establish are sharp.

In Section~\ref{junk} we investigate to what extent the decomposition parts containing a $k$-block can be required to contain nothing else. It turns out that there can be $k$-blocks that never occur as entire parts in a \td\ of adhesion~$<k$, due to a local obstruction in terms of the way in which these blocks are separated from the rest of~$G$. However, one can show that this is the only obstruction. We find a condition by which such blocks can be identified, which leads to the following best possible result: for every $k$, every finite graph has a canonical \td\ that distinguishes all its $k$-profiles efficiently and in which all those $k$-blocks are a part that occur as a part in some \td\ of adhesion~$<k$~\cite{CG14:isolatingblocks}. Finally, we establish some sufficient global conditions on $G$ to ensure that \emph{every\/} decomposition part containing a $k$-block contains nothing else. (So these conditions imply that there are no local obstructions to this as found earlier, for any $k$-block.)

In order to read this paper with ease, the reader should be familiar with~\partOne; in particular with the terminology introduced in Section~2 there, the notions of a task and a strategy as defined in Section~3, and the notion of a $k$-strategy as defined in Section~4. The proofs in~\partOne\ need not be understood in detail, but Examples 1 and~4 make useful background.

Readers interested in $k$-blocks as such may refer to~\cite{ForcingBlocks}, where we relate the greatest number $k$ such that $G$ has a $k$-block to other graph invariants.

Throughout this paper, we consider a fixed finite graph~$G = (V,E)$.

\section{Orientations of decomposition trees}\label{OrientingTDs}

By \cite[Theorem~2.2]{CDHH13CanonicalAlg}, every nested proper \sys~\cN\ of our graph $G = (V,E)$ gives rise to a \td~\TV\ that {\em induces\/} it, in that the  edges of the decomposition tree~\cT\ correspond to the separations in~\cN. How exactly \TV\ can be obtained from~\cN\ is described in~\cite{confing}. In this paper we shall be concerned with how profiles~-- in particular, blocks and tangles~-- correspond to nodes of~\cT. This correspondence will be injective~-- distinct blocks or tangles will `live in' distinct nodes of~\cT~-- but it will not normally be onto: only some of the nodes of~\cT\ will accommodate a block or tangle (all of some fixed order~$k$).

Our aim in this section is to show that every profile $P$ living in a node of $T$ defines a consistent orientation of~$E(T)$ (towards that node), or equivalently of~$N$ (namely,~$N\cap P$),%
   \footnote{Recall that profiles are consistent orientations of separation systems such as $N$ satisfying a further axiom~(P). A set of oriented tree-edges or of separations of~$G$ is \emph{consistent} if no two of them point away from each other; see~\cite{CDHH13CanonicalAlg} for the formal definition.}
   and that the set of \emph{all} consistent orientations of~$N$ corresponds bijectively to the nodes of~\cT. Let $\cV = (V_t)_{t\in\cT}$.

As \TV\ induces~\cN, there is for every separation $\AB\in\cN$ an oriented edge $e = t_A t_B  $ of~\cT\ such that, if $T_A$ denotes the component of $\cT - e$ that contains~$t_A $ and $T_B$ denote the component containing~$t_B  $, we have
 $$\AB = \Big( \bigcup_{t\in T_A}\! V_t\>,\bigcup_{t\in T_B}\! V_t\Big).$$
If \TV\ was obtained from~\cN\ as in \cite[Theorem~2.2]{CDHH13CanonicalAlg}, then $e$ is unique, and we say that it \emph{represents~\AB\ in~\cT\!}.%
   \footnote{In general if $e$ is not unique, we can make it unique by contracting all but one of the edges of $\cT$ inducing a given partition in~\cN, merging the parts corresponding to the nodes of contracted edges.}

Every node $t\in \cT$ induces an orientation of the edges of \cT\!, towards it. This corresponds as above to an orientation $O(t)$ of~\cN,
$$ O(t) := \{\AB\in \cN \mid t \in T_B \},$$
from which we can reobtain the part~$V_t$ of~\TV\ as
\begin{equation}
V_t\ =\!\!\! \bigcap_{\AB \in O(t)} \!\!\!B.\label{parts}
\end{equation}
We say that $t$ \emph{induces\/} the orientation $O(t)$ of~\cN, and that the separations in~$O(t)$ are \emph{oriented towards~$t$}.

Distinct nodes $t,t'\in\cT$ induce different orientations of~\cN, since these orientations disagree on every separation that corresponds to an edge on the path~$t\cT t'$. Also clearly, not all orientations of~\cN\ are induced by a node of~\cT. But it is interesting in our context to see which are:

\begin{thm}\label{living}
\begin{enumerate}[\rm (i)]\itemsep=0pt
\item The orientations of~\cN\ that are induced by nodes of~\cT\ are precisely the consistent orientations of~\cN.
\item An orientation of the set of all proper $(<k)$-separations of~$G$ orients the separations induced by any \td\ of adhesion~$<k$ towards a node of its decomposition tree if and only if it is consistent.
\end{enumerate}
\end{thm}

\proof
(i) Let $O$ be an orientation of~$\cN$ that is not induced by a node in $\cT\!$, and consider the corresponding orientation of (the edges of)~\cT. Then there are edges $e,e'$ of~\cT\ that point in opposite directions. Indeed, follow the orientated edges of~\cT\ to a sink~$t$; this exists since \cT\ is finite. As $t$ does not induce~$O$, some oriented edge $e'=t't''$ has $t$ lie in the component of $\cT-e'$ that contains~$t'$. Then $\cT-e'$ contains a $t'$--$t$ path. Its last edge $e$ is oriented towards~$t$, by the choice of~$t$. The separations $\AB,\CD\in O$ represented by $e$ and~$e'$ then satisfy $\BA \le \CD$, so $O$ is inconsistent. 

For the converse implication suppose an orientation $O(t)$ induced by some $t \in \cT$ is inconsistent. Then there are $\AB,\CD \in O(t)$ with $\DC \le \AB$.%
   \COMMENT{}
   Let $e$ be the oriented edge of~\cT\ representing~\AB, and let $f$ be the oriented edge representing~\CD.

Consider the subtrees $T_A,T_B,T_C,T_D$ of~\cT. Note that $T_B\cap T_D$ contains~$t$, by definition of~$O = O(t)$, and hence contains the component $T_t$ of $\cT-e-f$ containing~$t$.

If $f\in T_A$, then $T_B$ is a connected subgraph of~$\cT-f$ containing~$t$, and hence contained in~$T_D$. With $T_B\sub T_D$%
   \COMMENT{}
   we also have $B\sub D$. But now $\DC\le\AB$ implies $B\sub D\sub A$, and so \AB\ is not a proper separation. But it is, because $\AB\in\cN$. Hence $f\in T_B$, and similarly $e\in T_D$.

Let us show that $t_B, t_D\in T_t$. Suppose $f$ lies on the path in~\cT\ from $t$ to~$t_B$. Then this path traverses $f$ from $t_D$ to~$t_C$, since its initial segment from $t$ to~$f$ lies in~$T_D$ (the component of $\cT-f$ containing~$t$) and hence ends in~$t_D$. But then $e\in T_C$, contrary to what we have shown. Thus $f\notin t\cT t_B$, and clearly also $e\notin t\cT t_B$.%
   \COMMENT{}
   Therefore $t_B\in T_t$, and similarly $t_D\in T_t$.

Since $f\notin T_A$, we know that $T_A$ is a connected subgraph of~$\cT-f$ containing an end of~$e$. Adding $e$ to it we obtain a connected subgraph of $\cT-f$ that contains both ends of~$e$ and therefore meets~$T_t$, and adding $T_t$ too we obtain a connected subgraph of~$\cT-f$ that contains both $T_A$ and~$t_D$.%
   \COMMENT{}
   Therefore $T_A\sub T_D$, and thus $A\sub D$. Analogously, $C\sub B$. But now $\DC\le\AB$ implies both $A\sub D\sub A$ and $C\sub B\sub C$, giving $\AB=\DC$. But then $O$ contains both \CD\ and~\DC, which contradicts its definition as an orientation of~\cN.

(ii) If a given orientation of the set~$\cS_k$ of all proper $(<k)$-separations of~$G$ is consistent, then so is the orientation it induces on~\cN. By~(i), this orientation of~\cN\ orients it towards a node of the decomposition tree.

Conversely, if an orientation of~$\cS_k$ is inconsistent, then this is witnessed by separations $\AB,\CD\in\cS_k$ with $\CD\le\AB$ such that \AB\ is oriented towards~$B$ but \CD\ is oriented towards~$C$.%
   \COMMENT{}
   By \cite[Theorem~2.2]{CDHH13CanonicalAlg}, $\cN = \{\AB,\BA,\CD,\DC\}$ is induced by a \td~\TV. Since the orientation $\{\DC,\AB\}$ which our given orientation of $\cS_k$ induces on~\cN\ is inconsistent, we know from~(i) that it does not orient~\cN\ towards any node of~$\cT$.
\endproof

Theorem~\ref{living}\,(i) implies in particular that any profile $P$ which orients~\cN\ defines a unique node $t\in \cT$: the $t$ that induces its \cN-profile $P \cap \cN = O(t)$. We say that $P$ \emph{inhabits} this node~$t$ and the corresponding part~$V_t$. If $P$ is a $k$-block profile, induced by the $k$-block~$X$, say, then this is the case if and only if $X\sub V_t$.

Given a set \cP\ of profiles, we shall call a node~$t$ of~\cT\ and the corresponding part~$V_t$ \emph{essential (wrt.~\cP)} if there is a profile in \cP\ which inhabits $t$.

Nodes $t$ such that $V_t\sub A\cap B$ for some $\AB\in\cN$ are called \emph{hub nodes\/}; the node $t$ itself is then a \emph{hub}.
Example~2 in~\partOne\ shows that distinct hub nodes ${t, t' \in \cT}$ may have the same hub $V_t = V_t'$. So the bijection established by Theorem~\ref{living} does not induce a similar correspondence between the consistent orientations of~\cN\ and the parts of~\TV\ as a set, only as a family $\cV=(V_t)_{t\in\cT}$. This is illustrated by~\cite[Figure~7]{confing}.

Theorem~\ref{living}\,(ii) will not be needed in the rest of this paper. But it is interesting in its own right, in that it provides a converse to the following well-known fact in graph minor theory. Every haven~\cite{ST1993GraphSearching}, preference~\cite{ReedConnectivityMeasure}%
   \COMMENT{}
   or bramble~\cite{DiestelBook10noEE} of order~$\ge k$ in~$G$ orients the set~$\cS_k$ of all $(<k)$-separations of~$G$ (e.g., `towards' that bramble). In particular, it orients the separations induced by any \td\ of adhesion~$<k$, and it orients these towards a node of that decomposition tree. But this fact has no converse: while it is always possible to orient $\cS_k$ in such a way that the separations induced by any \td\ of adhesion~$<k$ are oriented towards a node~$t$ of the decomposition tree~-- take%
   \footnote{Every \sys\ has a consistent orientation; see~\partOne.}%
      \COMMENT{}
   any consistent orientation of~$\cS_k$ and apply Theorem~\ref{living}\,(ii)~-- this orientation of~$\cS_k$ need not be a haven or preference of order~$k$: there may be no bramble of order~$\ge k$ `living in'~$t$.%
   \footnote{For example, identify three copies of~$K^5$ in one vertex~$v$, and orient every ~$(<2)$-separation towards the side that contains two of these~$K^5$. This is a consistent orientation of~$\cS_2$ that is not a 2-haven or 2-preference and is not induced by a bramble of order~$\ge2$, but which still orients the 1-separations of any \td\ of adhesion~1 towards a node~$t$ (whose corresponding part could be either a $K^5$ or a $K^1$ hub).}

Theorem~\ref{living}\,(ii) shows that the consistent orientations of~$\cS_k$, which are generalizations of havens or preferences of order~$k$, are the unique weakest-possible such generalization that still orients all \td s of adhesion~$<k$ towards a node.

\section{Bounding the number of inessential parts}\label{sec_kcon}

Let $k\in\N$, and let \cP\ be a set of $k$-profiles of our graph~$G$, both fixed throughout this section. Whenever we use the term `essential' in this section, this will be with reference to this set~\cP.

Any canonical \td\ distinguishing~\cP\ has at least~$|\cP|$ essential parts, one for every profile in~\cP. Our aim in this section is to bound its number of inessential parts in terms of~$|\cP|$.

Variants of \cite[Example~1]{CDHH13CanonicalAlg} show that no such bounds exist if we ever use a strategy that has \smax, \srmax, \sext\ or~\sloc\ among its values, so we confine ourselves to strategies with values in~$\{\srext,\srloc\}$.%
   \COMMENT{}

The definition of the parts of a \td\ \TV\ being somewhat complicated (see Section~\ref{OrientingTDs}), rather than bounding the number $|\cV| - |\cP|$ of inessential parts of~\TV\ directly, we shall bound the number $|\cN|$ instead. Since ${1\over2}|\cN|$%
   \COMMENT{}
   is the number of edges of~\cT~-- as \cN\ contains `oriented' separations, every edge of \cT\ appears twice~-- and ${1\over2}|\cN|+1$ its number of nodes, the number of inessential parts will then be ${1\over2}|\cN|+1 - |\cP|$.

Our aim, then, will be to choose a strategy that minimizes~$|\cN|$. Our strategies should therefore take values in~$\{\srext,\srloc\}$ only, i.e., we should reduce our tasks before we tackle them, by deleting separations that do not distinguish any profiles in~\cP. Moreover, for a single reduced task~\SP\ we have $\sext\SP \sub \sloc\SP$ by~\cite[(11)]{CDHH13CanonicalAlg}, i.e., every separation chosen by~\sext\ is also chosen by~\sloc. This suggests that the overall strategy \sExt, which only uses~\srext, should also return fewer separations than~\sLoc, which only uses~\srloc~-- perhaps substantially fewer, since if we select fewer separations at each step we also have more interim steps in which we reduce.

Surprisingly, this is not the case. Although our general bounds for \sExt\ are indeed better than those for~\sLoc\ (or the same, which is already a surprise), Example~\ref{ex_locbetter} below will show \sLoc\ yields better results than~\sExt\ for some graphs.

Let $|\cP| =: p$. For single tasks~\SP, we obtain the following bounds on~$|\cN|$:

\begin{lem}\label{lem_probbounds}
For every feasible task \SP\ we have
\begin{align}
2(p-1) \le |\cN_{\sExt}\SP | &\le  2p \label{bd_Ext}, \mbox{ and} \\
2(p-1) \le |\cN_{\sLoc}\SP | &\le  4(p-1) \label{bd_Loc}.
\end{align}
\end{lem}
\begin{proof}
The two lower bounds, which in fact hold for any strategy, follow from the fact that \cN\ gives rise to a \td\ \TV\ that induces it and distinguishes~\cP: this means that $|\cN| = 2\,(|\cT|-1) \ge 2(p-1)$.

Let us now prove the upper bound in \eqref{bd_Ext}, by induction on~$p$. Let \RP\ be the reduction of \SP. If $p \le 1$ then $\cR=\es$, so the statement is trivial. Now assume that $p \ge 2$. Then $\cS\ne\es$, since \cS\ distinguishes~\cP. Let $\cP^\cE$ be the set of profiles in~\cP\ that are extremal in~\RP. By~\cite[Lemma~3.1]{CDHH13CanonicalAlg} we have $\cP^\cE \neq \emptyset$. Then
\begin{equation}
\cN_{\sExt}\SP =\ \cN\>\cup\!\bigcup_{O\,\in\,\cO_\cN}\!\! \cN_{\sExt} (\cS_O,\cP_O)\,,
\end{equation}
   by definition of~$\cN_{\sExt}\SP$, where $\cN = \srext\SP$. Every extremal ${P\in\cP}$ is distinguished from all the other profiles in~\cP\ by the separation \AB\ for which $P=P_{\AB}$,%
   \COMMENT{}
   so $P$ lies in a singleton class ${\cP_O = \{P\}}$. Then $(\cS_O,\cP_O)$ reduces to $(\es,\cP_O)$, giving ${\cN_{\sExt} (\cS_O,\cP_O) = \es}$ for these $O\in\cO_\cN$. By the uniqueness of~$P_{\AB}$ in \cite[Lemma~3.2]{CDHH13CanonicalAlg}, no separation in $\cN$ separates two non-extremal profiles from~\cP. So there is at most one other partition class~$\cP_O$ with $O\in\cO_\cN\,$.%
   \COMMENT{}
   If such a $\cP_O$ exists it satisfies $\cP_O = \cP\sm \cP^\cE$, and if it is non-empty the $O\in\cO_\cN$ giving rise to it is unique.%
   \COMMENT{}
   Therefore
   $$\cN_{\sExt}\SP = \srext\SP \cup \cN_{\sExt}(\cS_O,\cP\sm \cP^\cE)$$
for this~$O$ if $\cP\sm \cP^\cE\ne\es$, and $\cN_{\sExt}\SP = \srext\SP$ otherwise.%
   \COMMENT{}
   In the first case we have
   $$|\cN_{\sExt}(\cS_O,\cP\sm \cP^\cE)| \le\, 2\,|\cP \sm \cP^\cE|$$
by the induction hypothesis, and in both cases we have $|\srext\SP| \le 2\,|\cP^\cE|$ by \cite[Lemma~3.2 and~(9)]{CDHH13CanonicalAlg}. This completes the proof of~\eqref{bd_Ext}.

For a proof of the upper bound in~\eqref{bd_Loc} let \TV\ be a \td\ of~$G$ that induces $\cN_{\sLoc}\SP$ as in \cite[Theorem~2.2]{CDHH13CanonicalAlg}. Since $\cN_{\sLoc}\SP$ contains only \cP-relevant separations,%
   \COMMENT{}
   all the leaves of \cT\ are essential. Furthermore, we shall prove the following:
\begin{txteq}\label{eq_oneendess}
For every edge $e = t_1t_2$ of \cT, either $t_1$ or $t_2$ is essential.
\end{txteq}

Before we prove~\eqref{eq_oneendess}, let us show how it helps us establish the upper bound in~\eqref{bd_Loc}. If \eqref{eq_oneendess} holds, then all the neighbours of an inessential node are essential. Let $\cT'$ be obtained from~\cT\ by deleting each inessential node and adding an edge from one of its neigbours to all its other neighbours. Let us show that
\begin{txteq}\label{Tprime}
$\cT'$ has $p$ nodes and at least half as many edges as~$\cT$.
\end{txteq}
The first of these assertions holds by definition of \cT\ and~$p$. For the second, note that for each inessential node we delete we lose exactly one edge. So to prove the second claim in~\eqref{Tprime} it suffices to show that $\cT$ has at most $\size\cT/2$ inessential nodes. But this follows from~\eqref{eq_oneendess} and the fact that the leaves of \cT\ are essential: every inessential node has at least two incident edges, and no edge is counted twice in this way (i.e., is incident with more than one inessential node).

By~\eqref{Tprime}, \cT\ has at most $2(p-1)$ edges. Since $\cN_{\sLoc}\SP$ is induced by \TV, this will establish the upper bound in~\eqref{bd_Loc}.

So let us prove~\eqref{eq_oneendess}. Suppose \cT\ has an edge $e = t_1t_2$ with neither $t_i$ essential. Let $\AB \in \cN_\sLoc\SP$ be the separation which $e$ induces. Let $T_A$ denote the component of $\cT - e$ that contains~$t_1$, and let $T_B$ be the component containing~$t_2$. 

At the time \AB\ was chosen by $\sLoc$ we had a nested proper \sys~\cN\ and a consistent orientation $O$ of~\cN\ such that $\AB \in \srloc(\cS_O, \cP_O)$. (When $\cN=\es$ at the start, we have $\srloc(\cS_O,\cP_O) = \SP$.) So there is a profile $P \in \cP_O$ such that \AB\ or \BA\ is maximal in $(P \cap \cS_O, \le)$, say~\AB.%
   \COMMENT{}
By the definition of a task, $P$ orients $\cS$. By Lemma~\ref{living}, therefore, $P$~inhabits a unique node $t\in \cT$, making it essential. Then $\AB \in O(t)$, and hence $t \in T_B$. Since $t_2$ inessential by assumption, $t \neq t_2$. 

The last edge $e'$ on the $t_2$--$t$ path in~\cT\ induces a separation ${\CD \in O(t) \sub P}$, and $\AB \le \CD$, or equivalently, $\DC \le \BA$. Since \AB\ is $\cP_O$-relevant there exists $P' \in \cP_O$ with $\BA \in P'$. Then $\DC \in P'$, since $P'$ orients all of $\cS$ consistently. But then \CD\ splits~$O$, and thus lies in~$\cS_O$. This contradicts the maximality of~\AB, completing the proof of~\eqref{eq_oneendess} and hence of~\eqref{bd_Loc}.
\end{proof}

It is easy to see that the upper bounds in Lemma~\ref{lem_probbounds} are tight. For example, if $G$ consists of $n$ disjoint large complete graphs threaded on a long path, then for $k=3$ the canonical \td\ produced by~\sLoc\ will have $n$ essential parts consisting of these complete graphs and $n-1$ inessential parts consisting of the paths between them. When $n$ is even, this example also shows that the upper bound for~\sExt\ is best possible. In fact, the following example shows that the upper bound in Lemma~\ref{lem_probbounds}\,(i) is best possible for all canonical \td s (regardless of which strategy is used to produce it):

\begin{ex}\label{cycle}
Let $G$ consist of an $n$-cycle $C$ together with $n$ large complete graphs $K_1,\dots,K_n$ each intersecting $C$ in one edge and otherwise disjoint. Then, for $n\ge 3$ and $k=3$, any canonical \td\ of $G$ of adhesion~${<k}$ either has exactly one part or exactly the parts $C,K_1,\dots,K_n$. This is because the 2-separations of~$G$ that induce 2-separations of~$C$ cannot be induced by a canonical \td\ of~$G$, since they cross their translates under suitable automorphisms of~$G$.
\end{ex}

It is more remarkable, perhaps, that the upper bound in Lemma~\ref{lem_probbounds}\,(i) is so low: that at most one part of the \td\ is not inhabited by a profile from~$\cP$. For if $G$ is $(k-1)$-connected and \cS\ is the set of all proper $(<k)$-separations, then every task \SP\ is feasible \cite[Lemma~4.1]{CDHH13CanonicalAlg}, and hence Lemma~\ref{lem_probbounds} gives the right overall bounds.

If $G$ is not $(k-1)$-connected, the original task \SP\ need not be feasible, and we have to use iterated strategies. Let $\sExt^k$ denote the $k$-strategy all whose entries are~$\sExt$, and let $\sLoc^k$ denote the $k$-strategy which only uses $\sLoc$. Interestingly, having to iterate costs us a factor of~2 in the case of~\sExt, but it does not affect the upper bound for~\sLoc. Hence for iterated strategies the two bounds coincide:

\begin{thm}\label{thm_itbounds}
Let \cP\ be any set of $k$-profiles of $G$, and $p := |\cP|$. Let $\cN_{\sExt^k}(\cP)$ and $\cN_{\sLoc^k}(\cP)$ be obtained with respect to the set \cS\ of all proper $(<k)$-separations.
\begin{enumerate}[\rm(i)]
\item $2(p-1) \le |\cN_{\sExt^k}(\cP) | \le  4\,(p-1)$
\item $2(p-1) \le |\cN_{\sLoc^k}(\cP) | \le  4\,(p-1)$
\item If $G$ is $(k-1)$-connected, then $|\cN_{\sExt^k}(\cP) | \le  2p$.
\end{enumerate}
\end{thm}

\medbreak\proof
The lower bounds for~\cN\ follow as in the proof of Lemma~\ref{lem_probbounds}. Statement (iii) reduces to Lemma~\ref{lem_probbounds}\,(i), since $\sExt^k = \sExt$ now and the entire task~\SP\ is feasible \cite[Lemma~4.1]{CDHH13CanonicalAlg}.

For the proof of the upper bounds in (i) and~(ii), let us define a rooted tree $(T,r)$ that represents the recursive definition of $\cN_{\sExt^k}$ and $\cN_{\sLoc^k}$, as follows. Let 
 $$V(T) := \{\emptyset\}\cup\!\! \bigcup_{1\le\ell\le k}\!{\cP_\ell}\,;$$ 
recall that $\cP_\ell$ for $\ell\le k$ is the set of all $\ell$-profiles of~$G$ that extend to a $k$-profile in~$\cP$. We select $r = \emptyset$ as the root, and make it adjacent to every $P\in\cP_1$. For $2\le\ell\le k$ we join $P\in\cP_\ell$ to the unique $P'\in\cP_{\ell-1}$ which it induces (i.e., for which $P'\sub P$). This is clearly a tree, with levels $\{\es\},\cP_1,\dots,\cP_k$. Let us call the vertices of $T$ that are not in~$\cP_k$ its \emph{internal} vertices.

The internal vertices of $T$ correspond bijectively to the tasks which our iterated algorithm, either $\sExt^k$ or~$\sLoc^k$, has to solve. Indeed, at the start the algorithm has to solve the task~\SPp\ with $\cS'$ the set of proper 0-separations of $G$ and $\cP'=\cP_1$ the set of 0-profiles that extend to a $k$-profile in~$\cP$. This task corresponds to~$r$ in that $\cP'$ is the set of children of~$r$. Later, for $\ell=2,\dots,k$ recursively, the algorithm at step~$\ell$ receives as input some tasks \SPp, one for every $P\in\cP_{\ell-1}$, in which $\cP'$ is the set of $\ell$-profiles in~$\cP_\ell$ extending~$P$, and $\cS'$ is the set of proper $(\ell-1)$-separations of $G$ that are nested with the set $\cN_{\ell-1}$ of nested $(<\ell-1)$-separations distinguishing~$\cP_{\ell-1}$ which the algorithm has found so far. This task corresponds to $P\in V(T)$ in the same way, in that $\cP'$ is the set of children of~$P$.

Let $c(v)$ denote the number of children of an internal vertex $v$. Since \sExt\ and \sLoc\ reduce every task before they solve it, the task \SPp\ corresponding to a vertex $v$ will add a separation to~\cN\ only if $c(v) = |\cP'| \ge 2$. Let $(T',r')$ be obtained from~$(T,r)$ by suppressing any vertices with exactly one child; if $r$ is suppressed, its first descendant with more than one child becomes the new root~$r'$. The internal vertices of $T'$ thus have degree at least~$3$, except that $r'$ has degree at least~$2$. Let $i$ denote the number of internal vertices of~$T'$. Since the number of (non-root) leaves of $T'$ is exactly~$p$, we have at most $(p-1)$ internal vertices, that is, $i \le p-1$.%
   \COMMENT{}

Now consider the construction of $\cN_{\sExt^k}(\cP)$. By \eqref{bd_Ext} in Lemma~\ref{lem_probbounds}, each internal vertex~$v$ of~$T'$ contributes at most $2c(v)$ separations. So there are at most twice as many separations in $\cN_{\sExt^k}(\cP)$ as there are edges in $T'$:
   $$|\cN_{\sExt^k}(\cP)| \le 2\,\size{T'} = 2(p + i -1) \le 4(p-1).$$

During the construction of $\cN_{\sLoc^k}(\cP)$, each internal vertex~$v$ of~$T'$ contributes at most $4(c(v) -1)$ separations, by \eqref{bd_Loc} in Lemma~\ref{lem_probbounds}. Writing $I$ for the set of internal vertices of~$T'$, we thus obtain%
   \COMMENT{}
$$|\cN_{\sLoc^k}(\cP)| \le 4\sum_{v\in I} \big(c(v)-1\big) = 4\big(\size{T'} -\> i\big) = 4\big( |T|- i - 1) = 4(p-1).\qed$$

It is easy to construct examples showing that all these bounds are sharp. Instead, let us give an example where $\sLoc$ yields the best possible result of $2(p-1)$, while $\sExt$ does not:

\begin{ex}\label{ex_locbetter}
Consider the $3$-connec\-ted graph with four $4$-blocks shown in Figure~\ref{fig_locbetter}. The grey bars indicate separators of chosen separations. Algorithm \sExt\ chooses all these separations: first the two pairs of outer separations, then the two pairs of inner separations. On the other hand, \sLoc\ will choose the three pairs of `straight' separations at the first step, and no further separations thereafter. Therefore \sExt\ chooses one pair of separations more than \sLoc\ does.
\begin{figure}[h]
\begin{center}
\includegraphics[width=.45\textwidth]{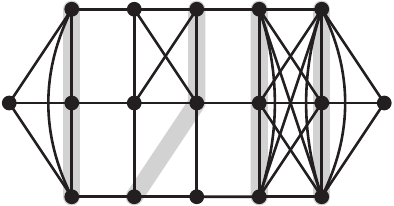}
\caption{A graph where \sLoc\ chooses fewer separations than \sExt.}\label{fig_locbetter}
\end{center}\vskip-12pt
\end{figure}
\end{ex}

\section{Bounding the size of the parts}\label{junk}

One of the first questions one may ask about canonical \td s is whether they can be chosen so as to witness the tree-width of the graph. Choosing the cycle $C$ in Example~\ref{cycle} long, however, shows that this will not in general be the case: restricting ourselves to a set of separations that is invariant under all the automorphisms of~$G$ can result in arbitrarily large parts, and these need not even be essential.%
   \COMMENT{}

However, if we restrict our attention from arbitrary $k$-profiles to (those induced by) $k$-blocks, we can try to make the essential parts small by reducing the junk they contain, the vertices contained in an essential part that do not belong to the $k$-block that made this part essential. Note that this aim conflicts with our earlier aim to reduce the number of inessential parts: since this junk is part of~$G$, expunging it from the essential parts will mean that we have to have other parts to accommodate it.

In general, we shall not be able to reduce the junk in essential parts to zero unless we restrict the class of graphs under consideration. Our next example shows some graphs for which any \td\ of adhesion at most~$k$, canonical or not, has essential parts containing junk. The amount of junk in a part cannot even be bounded in terms the size of the $k$-block inhabiting it.\looseness=-1

\begin{ex}\label{ex_junky}
Consider the 4-connected graph obtained by joining two adjacent vertices $x,y$ to a $K^5$ as in Figure~\ref{fig_junky}. This graph has a single $5$-block~$K$, the vertex set of the~$K^5$. In any \td\ of adhesion at most~$4$, the part containing $K$ will contain $x$ or $y$ as well: since the 4-separations that separate $x$ and $y$ from~$K$ cross, at most one of them will be induced by the decomposition.

\begin{figure}[h]
\begin{center}
\includegraphics{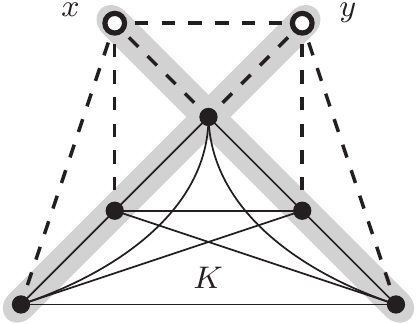}
\caption{A $K^5$ with unavoidable junk attached}\label{fig_junky}
\end{center}\vskip-12pt
\end{figure}

To increase the amount of junk in the part containing~$K$, we can attach arbitrarily many pairs of adjacent vertices to the~$K^5$ in the same way as we added $x$ and~$y$. This will not increase the size of the $5$-block~$K$, but the part containing~$K$ will also contain at least one vertex from each of those pairs.
\end{ex}

The following theorem shows that the obstruction to obtaining essential parts without junk illustrated by the above example is, in a sense, the only such obstruction. Let us call a $k$-block $X$ \emph{well separated\/} in a \sys~\cS\ of proper $(<k)$-separations if the $k$-profile $P_k(X)\cap\cS$ it induces in~\cS\ is well separated, that is, if the maximal elements of $P_k(X)\cap\cS$ are nested with each other. (This fails in the example.) Recall that a separation \AB\ is \emph{tight} if every vertex in $A\cap B$ has a neighbour in~$A\sm B$ and a neighbour in~$B\sm A$.

\begin{thm}\label{thm_leafnice}
Let $1\le k\in\N$, and let \cS\ be a set of proper $(<k)$-separations that includes all the tight $(<k)$-separations.%
   \COMMENT{}
   Then every graph $G$ has a canonical\,%
   \footnote{Here, this means that the \td\ will be invariant under those automorphisms of~$G$ that act on~\cS. For example, this is the case for all the automorphisms of~$G$ if \cS\ consists of all the tight $(<k)$-separations.}
    \td\ all whose separations induced by tree-edges are in~\cS\ such that
\begin{enumerate}[\rm (i)]\itemsep=0pt\vskip-3pt\vskip0pt
\item distinct $k$-blocks lie in different parts;
\item parts containing a $k$-block that is well separated in~\cS\ coincide with that $k$-block;
\item if the task \SP\ with \cP\ the set of all $k$-block profiles is reduced and feasible,\penalty-200 then every leaf\,%
   \COMMENT{}
   part is a $k$-block.%
   \footnote{Recall that \SP\ is feasible, for example, if $G$ is $(k-1)$-connected.}%
   \COMMENT{}
\end{enumerate}
Every such decomposition that satisfies~{\rm (i)}, but not necessarily {\rm (ii)} or~{\rm (iii)}, can be refined to such a \td\ that also satisfies~{\rm (ii)} and~{\rm (iii)}.
\end{thm}

\begin{proof}
Let $\cP$ be the set  of all $k$-block profiles in~$G$. Let $\cN\sub\cS$ be any nested \sys\ that distinguishes all its $k$-blocks and is canonical, i.e., invariant under the auto\-morphisms of~$G$. Such a set $\cN$ exists by \cite[Theorem~4.4]{CDHH13CanonicalAlg}. (The separations provided by that theorem are tight, and hence in~\cS, because they are \cP-essential, i.e., distinguish two profiles in~\cP\ `efficiently'; see~\cite{CDHH13CanonicalAlg}.) Then the \td\ \TV\ that induces $\cN$ by \cite[Theorem~2.2]{CDHH13CanonicalAlg} satisfies~(i).

For~(ii) we refine \cN\ by adding the locally maximal separations of~\SP\ and their inverses.%
   \COMMENT{}
   These are nested with~\cS\ by \cite[Corollary~3.5]{CDHH13CanonicalAlg}. Hence the refined \sys\ $\cN\,'$ is again nested, and therefore induced by a \td\ $(\cT',\cV')$. This decomposition is again canonical, since the set of locally maximal separations is invariant under the automorphisms of~$G$. Clearly, $(\cT',\cV')$ still satisfies~(i).

To show that $(\cT',\cV')$ satisfies~(ii), suppose it has a part that contains a well separated $k$-block $X$ and a vertex $v$ outside $X$. By the maximality of~$X$ as a ${(<k)}$-inseparable set, there is a separation $(A,B)\in \cS$ with $X\sub B$ and ${v\in A\sm B}$. Clearly, $\AB\in P_k (X)$; choose $(A,B)$ maximal in~$P_k (X)$, the $k$-profile that $X$ induces. Then $\AB\in\cN\,'$, by definition of~$\cN\,'$. This contradicts our assumption that $v$ lies in the same part of $(\cT',\cV')$ as~$X$.

To show that the decomposition $(\cT',\cV')$ obtained for~(ii) also satisfies~(iii), consider a leaf part~$V_t$. By the assumption in~(iii), the separation $\AB\in\cN\,'$ that corresponds to the edge of $\cT'$ at~$t$ and satisfies $B=V_t$ distinguishes two $k$-blocks. Let $X$ be the $k$-block in~$V_t$; it is unique, since $\cN\,'$ distinguishes~\cP\ but no separation in~$\cN\,'$ separates~$V_t$.%
   \COMMENT{}
   Let $P=P_k(X)$ be the $k$-profile that $X$ induces. Let $(A',B')\ge\AB$ be maximal in~$\cS$. By assumption in~(iii), $B'\sub B$ too contains a $k$-block, which can only be~$X$. Hence $(A',B')\in P$.

Since \SP\ is reduced and feasible, by assumption in~(iii), \cite[Lemma~3.1]{CDHH13CanonicalAlg} implies that $(A',B')$ is extremal in~\cS. Hence~$P$ is extremal,%
   \COMMENT{}
   and therefore well separated. We thus have $V_t=X$ by~(ii).
\end{proof}

The idea behind allowing some flexibility for \cS\ in Theorem~\ref{thm_leafnice} is that this can make (ii) stronger by making more $k$-blocks well separated. For example, consider a \td\ whose parts are all complete graphs~$K^5$ and whose separations induced by tree edges all have order~3. The $k$-blocks for $k=5$ are the $K^5$s, but none of these is well-separated in the set \cS\ of all proper $(<k)$-separations, since the natural 3-separations can be extended in many ways to pairwise crossing 4-separations that will be the locally maximal separations. However all the $k$-blocks are separated in the smaller set $\cS'$ of all proper $(<4)$-separations, which are precisely the tight $(<k)$-separations. So applying the theorem with this $\cS'$ would exhibit that the essential parts of our decomposition are in fact $k$-blocks, a fact the theorem applied with~\cS\ cannot see. 

However, even with $\cS$ the set of tight $(<k)$-separations, Theorem~\ref{thm_leafnice}\,(ii) can miss some parts in canonical \td s that are in fact $k$-blocks, because they are not well separated even in this restricted~\cS:

\begin{ex}\label{BadSepkBlockEx}
Let $G$ consist of a large complete graph $K$ to which three further large complete graphs are attached: $K_1$ and~$K_2$ by separators $S_1$ and~$S_2$, respectively, and $K_{12}$ by the separator $S_1\cap S_2$. If $|S_1| = |S_2| = k-1$ and $S_1\ne S_2$,%
   \COMMENT{}
   the separations $(K_1\cup K_{12}, K\cup K_2)$ and $(K_2\cup K_{12}, K\cup K_1)$ are maximal in $P_k(K)\cap\cS$ for the $k$-block $K$ and the set \cS\ of all tight $(<k)$-separations. They cross, since both have $K_{12}$ on their `small' side (Fig.~\ref{BadSepkBlockFig}).

   \begin{figure}[htpb]
\centering
   	  \includegraphics{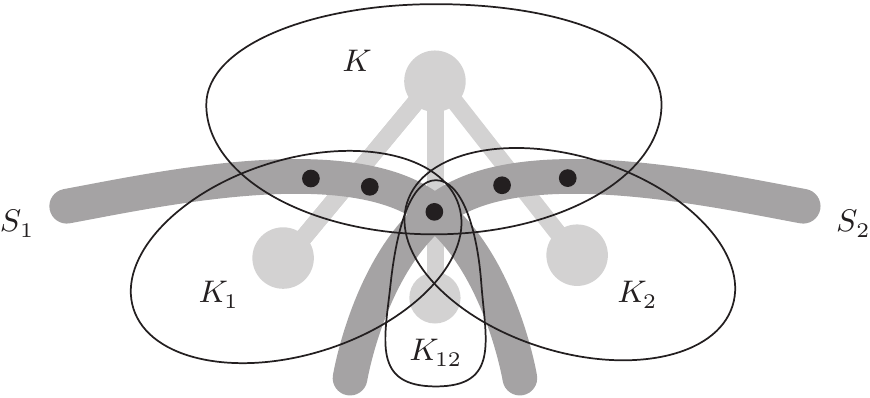}
   	  \caption{The 4-block $K$ is a decomposition part but is not well-separated}
   \label{BadSepkBlockFig}\vskip-3pt\vskip0pt
   \end{figure}

\goodbreak

So $K$ is not well separated. But the (unique) canonical \td\ of $G$ that distinguishes its $k$-blocks still has $K$ as a part: its parts are the four large complete graphs, the decomposition tree being a star with centre~$K$.
\end{ex}

We wonder whether the notion of being well separated can be weakened,%
   \COMMENT{}
   or applied to a suitable set $\cS$ of $(<k)$-separations, so as to give Theorem~\ref{thm_leafnice}\,(ii) a converse: so that every graph has a canonical \td\ that distinguishes its $k$-blocks, whose separations induced by decomposition tree edges are in~$\cS$, and in which every well separated $k$-block is a part, while conversely every $k$-block that occurs as a part in such a \td\ is well separated in~$\cS$.%
   \COMMENT{}

Here is an attempt. Given a $k$-block~$X$, let $\cS(X)$ denote the set of all tight separations \AB\ such that $X\sub B$ and $A\sm B$ is a component of~$G-X$.%
   \COMMENT{}
   The separations in~$\cS(X)$ are clearly nested.%
   \footnote{However, Example~\ref{BadSepkBlockEx} with $X=K_1$ and $X'=K_2$ shows that for distinct $k$-blocks~$X,X'$ the sets $\cS(X)$ and $\cS(X')$ need not be nested: the separation $\AB\in\cS(X)$ with $A=K\cup K_2$ and $B= K_1\cup K_{12}$ crosses the separation $(A',B')\in\cS(X')$ with $A'=K\cup K_1$ and $B= K_2\cup K_{12}$.}
    Write $\cS_k$ for the set of all $(<k)$-separ\-ations of~$G$. Then the condition that
\begin{equation}
\cS(X)\sub\cS_k\label{SXSk}
\end{equation}
   is a weakening of $X$ being well-separated in~$\cS_k$. Indeed, if $\AB\in\cS(X)$ is not in~$\cS_k$, i.e.\ has order~$\ge k$, we can find two crossing separations both maximal in~$P_k(X)$, as follows. Pick a vertex $a\in A\sm B$. By the maximality of $X$ as a $(<k)$-insep\-arable set, our vertex~$a$ can be separated from~$X$ by some $\CD\in\cS_k$, say with $a\in C\sm D$ and $X\sub D$. Then $\CD\in{P_k(X)}$. Replace \CD\ with any maximal separation in~${P_k(X)}$ that is greater than it, and rename that separation as~\CD. Then still $a\in C\sm D$ and $X\sub D$. As $A\sm B$ is connected, and $A\cap B\sub X$ has size~$\ge k$ although \AB\ is tight, it follows that $C\cap D$ contains a vertex~$a'\in A\sm B$. Like~$a$, the vertex $a'$ is separated from $X$ by some maximal separation $(C',D')\in{P_k(X)}$. The separations \CD\ and $(C',D')$  are easily seen to cross,%
   \COMMENT{}
   so $X$ is not well separated.

On the other hand, condition~\eqref{SXSk} holds for every $k$-block $X$ that does occur as a part in a \td\ of adhesion~$<k$. Thus if \eqref{SXSk} is still strong enough to imply that $X$ is a part in some, or any, canonical such \td, we shall have our desired converse of Theorem~\ref{thm_leafnice}\,(ii) with \eqref{SXSk} replacing `well separated'.

Given~$k$, call a \td\ of a graph \emph{good} if it is canonical, has adhesion~$<k$,%
   \COMMENT{}
   and distinguishes all the $k$-blocks of~$G$ \emph{efficiently}: any two of them are separated by an adhesion set whose order is minimum among all the sepators in $G$ between those blocks. The following result, which had been conjectured in the original preprint of this paper, was announced by Carmesin and Gollin~\cite{CG14:isolatingblocks}:

\begin{thm}\label{7} For every $k$, every finite graph has a good \td\ in which every $k$-block $X$ that satisfies~\eqref{SXSk} is a part.
\end{thm}

The proof of Theorem~\ref{7} is quite involved and builds on~\cite[Theorem~5.2]{confing}. However, the theorem has a corollary that can be stated with a minimum of technical overheads and emphasises the way in which it is best possible:

\begin{cor}
For every $k$, every finite graph has a good \td\ that includes among its parts all $k$-blocks that are a part in some \td\ of adhesion~$<k$.
\end{cor}

\begin{proof}
Consider the \td~\TV\ provided by Theorem~\ref{7}. Let $X$ be any $k$-block that occurs as a part in some good \td. As noted earlier, this implies that $X$ satisfies~\eqref{SXSk}. By the choice of~\TV, this means that $X$ is also a part of~\TV.
\end{proof}
   \COMMENT{}

Are there any natural conditions ensuring that \emph{every\/} essential part is a $k$-block? (In particular, such conditions will have to rule out Example~\ref{ex_junky}.) We do not know the answer to this question. But we can offer the following:

\begin{thm}\label{cond_well-sep}
Assume that $G$ is $(k-1)$-connected, and that 
every pair $x,y$ of adjacent vertices has one of the following properties:\vskip-3pt\vskip0pt
\begin{enumerate}[\rm (i)]\itemsep=0pt
\item $x$ and~$y$ have at least $k-3$ common neighbours;
\item $x$ and~$y$ are joined by at least $\lfloor\frac{3}{2}(k-2)\rfloor$ independent paths other than~$xy$;
\item $x$ and~$y$ lie in a common $k$-block.
\end{enumerate}
Then $G$ has a canonical \td\ of adhesion $< k$
such that every part containing a $k$-block is a $k$-block. In particular distinct $k$-blocks are contained in different parts.
\end{thm}

\begin{proof}
By Theorem~\ref{thm_leafnice} it suffices to show that every element $P$ of the set \cP\ of $k$-block profiles is well separated in the set \cS\ of all the proper $(<k)$-separations.

We do this by applying \cite[Lemma~3.4]{CDHH13CanonicalAlg}. Given $P\in\cP$, let $(A,B),(C,D)$ be crossing separations \
in $P \cap \cS$. If the separation $(A\cup C,B\cap D)$ has order~$\le k-1$, it is in $P\cap \cS$ by~(P2), and we are done. If not, then the separation $(B\cup D,A\cap C)$ has order~$<(k-1)$. Since $G$ is $(k-1)$-connected, $(B\cup D,A\cap C)$ must be improper. This means that $A\cap C \sub B\cup D$, because $B\not\sub A$ as $B$ contains a $k$-block.%
   \COMMENT{}
   But since \AB\ and \CD\ cross, we cannot have $A\cap C \sub B\cap D$. By symmetry we may assume that there is a vertex $x\in (C\cap D) \sm B$. As $G$ is $(k-1)$-connected, \CD\ is tight, so $x$ has a neighbour $y\in(A\cap B)\sm D$. Let $e := xy$.

Suppose first that $e$ satisfies (i). Since all common neighbours of $x$ and $y$ lie in $A\cap C$, this implies $k-1\leq|A\cap C|\le k-2$,%
   \COMMENT{}
   a contradiction.

Now suppose that $e$ satisfies~(ii), and let $\Wcal$ be a set of at least $\lfloor\frac{3}{2}(k-2)\rfloor$ independent 
$x$--$y$ paths other than the edge~$xy$. Let
 $$X:= (A\cap C)\sm\{x,y\}\qquad Y:= (A\cap B)\sm C\qquad Z:= (C\cap D)\sm A\,.$$
Since $A\cap C \sub B\cup D$, we have
 \begin{equation}\label{XYZ}
|X| + |Y| + |Z| \le |A\cap B|-1 + |C\cap D|-1 = 2(k-2).
\end{equation}
Every path in \cW\ that avoids $X$ meets both $Y$ and~$Z$.%
  \COMMENT{}
   As $|X|\le (k-2)-2$, this yields
   \COMMENT{}
 $$|\cW|\le |X| + {\textstyle{1\over2}}\big( |Y| + |Z|\big) \le |X| + (k-2)-\textstyle{1\over2}|X| \le \frac{3}{2}(k-2) - 1,$$
a contradiction.

Finally assume that $e$ satisfies (iii). Let $X$ be a $k$-block containing $x$ and~$y$.
As $x\notin B$ and $y\notin D$ we have $X\sub A\cap C$, contradicting $|A\cap C|\le k-2$.
\end{proof}

For $k=2$, Theorem~\ref{cond_well-sep}\,(i) implies Tutte's theorem that every 2-connected graph has a \td\ whose essential parts are precisely its 3-blocks. The decomposition obtained by any strategy starting with~$\smax$ is the decomposition provided by Tutte~\cite{TutteGrTh}, in which the inessential parts have cycle torsos.%
   \COMMENT{}

\bibliographystyle{plain}
\bibliography{collective}

\end{document}

%% file: defs.tex
\usepackage{amsmath, amsthm, amssymb, amstext, amsfonts}
\usepackage{enumerate}
\usepackage[applemac]{inputenc}
\usepackage{graphicx}

\theoremstyle{plain}
\newtheorem{thm}{Theorem}[section]
\newtheorem{lem}[thm]{Lemma}

\newtheorem{cor}[thm]{Corollary}

\theoremstyle{definition}

\newtheorem{ex}{Example}

\newcount\commentno
\def\COMMENT#1{$^{<\the\commentno>}$%
     \vadjust{\vbox to 0pt{\vss\vskip-8pt\rightline{%
     \rlap{\hbox{\hskip7mm \vbox{\pretolerance=-1
     \doublehyphendemerits=0 \finalhyphendemerits=0
     \hsize40mm\tolerance=10000\eightpoint
     \lineskip=0pt\lineskiplimit=0pt
     \rightskip=0pt plus16mm\baselineskip8pt\noindent
     \hskip0pt       
     {$\langle$\the\commentno. #1$\rangle$}\endgraf}}}}\vss}}%
     \global\advance\commentno by1}%
\def\writecommentsasfootnotes{%
 \def\COMMENT{\global\advance\commentno by1\footnote{$^{<\the\commentno>}$}}%
 }
\def\nocomments{\def\COMMENT##1{}}
\def\?#1{\vadjust{\vbox to 0pt{\vss\vskip-8pt\leftline{%
     \llap{\hbox{\vbox{\pretolerance=-1
     \doublehyphendemerits=0\finalhyphendemerits=0
     \hsize16truemm\tolerance=10000\small
     \lineskip=0pt\lineskiplimit=0pt
     \rightskip=0pt plus16truemm\baselineskip8pt\noindent
     \hskip0pt        
     #1\endgraf}\hskip7truemm}}}\vss}}}
\newenvironment{txteq}
  {
    \begin{equation}
    \begin{minipage}[c]{0.85\textwidth} 
    \em                                
  }
  {\end{minipage}\end{equation}\ignorespacesafterend}
\newenvironment{txteq*}
  {
    \begin{equation*}
    \begin{minipage}[c]{0.85\textwidth} 
    \em                                
  }
  {\end{minipage}\end{equation*}\ignorespacesafterend}
\def\specrel#1#2{\mathrel{\mathop{\kern0pt #1}\limits_{#2}}}
\def\Specrel#1#2{\mathrel{\mathop{\kern0pt #1}\limits^{#2}}}

\newcommand{\N}{\ensuremath{\mathbb{N}}}

\newcommand{\sm}{\ensuremath{\smallsetminus}}

\newcommand{\es}{\ensuremath{\emptyset}}

\newcommand{\sub}{\subseteq}

\newcommand{\cE}{\ensuremath{\mathcal E}}

\newcommand{\cN}{{\ensuremath N}}
\newcommand{\cO}{\ensuremath{\mathcal O}}
\newcommand{\cP}{\ensuremath{\mathcal P}}

\newcommand{\cR}{{\ensuremath R}}
\newcommand{\cS}{{\ensuremath S}}
\newcommand{\cT}{{\ensuremath T}}

\newcommand{\cV}{\ensuremath{\mathcal V}}
\newcommand{\cW}{\ensuremath{\mathcal W}}

\newcommand{\sfont}[1]{\ensuremath{\mathsf{#1}}}

\newcommand{\srmax}{\sfont{all_r}}
\newcommand{\srloc}{\sfont{loc_r}}
\newcommand{\srext}{\sfont{ext_r}}

\newcommand{\smax}{\sfont{all}}
\newcommand{\sloc}{\sfont{loc}}
\newcommand{\sext}{\sfont{ext}}

\newcommand{\sLoc}{\sfont{Loc}}
\newcommand{\sExt}{\sfont{Ext}}

\newcommand{\SP}{\ensuremath{(\cS,\cP)}}

\newcommand{\RP}{\ensuremath{(\cR,\cP)}}

\newcommand{\SPp}{\ensuremath{(\cS',\cP')}}

\newcommand{\TV}{\ensuremath{(\cT,\cV)}}

\def\td{tree-decom\-po\-si\-tion}

\newcommand{\sys}{separation system}

\newcommand{\sepn}[2]{\ensuremath{{(#1,#2)}}}
\newcommand{\AB}{\sepn AB}
\newcommand{\BA}{\sepn BA}
\newcommand{\CD}{\sepn CD}
\newcommand{\DC}{\sepn DC}

\let\doublebar=\| \def\size#1{{\doublebar#1\doublebar}}

\newcommand{\mcm}[3]{\newcommand{#1}[#2]{{\ensuremath{#3}}}} 

\mcm{\tuple}{1}{\langle #1 \rangle}
\mcm{\name}{1}{\ulcorner #1 \urcorner}
\mcm{\Nbb}{0}{\mathbb{N}}
\mcm{\Zbb}{0}{\mathbb{Z}}
\mcm{\Rbb}{0}{\mathbb{R}}
\mcm{\Cbb}{0}{\mathbb{C}}
\mcm{\Fbb}{0}{\mathbb{F}}
\mcm{\Bcal}{0}{\cal B}
\mcm{\Ccal}{0}{\cal C}
\mcm{\Dcal}{0}{\cal D}
\mcm{\Ecal}{0}{\cal E}
\mcm{\Fcal}{0}{\cal F}
\mcm{\Gcal}{0}{\cal G}
\mcm{\Hcal}{0}{\cal H}
\mcm{\Ical}{0}{\cal I}
\mcm{\Lcal}{0}{\cal L}
\mcm{\Mcal}{0}{\cal M}
\mcm{\Ncal}{0}{\cal N\!}
\mcm{\Ocal}{0}{\cal O}
\mcm{\Pcal}{0}{{\cal P}}
\mcm{\Scal}{0}{{\cal S}}
\mcm{\Tcal}{0}{{\cal T}}
\mcm{\Ucal}{0}{{\cal U}}
\mcm{\Vcal}{0}{{\cal V}}
\mcm{\Wcal}{0}{{\cal W}}
\mcm{\Ycal}{0}{{\cal Y}}
\mcm{\Mfrak}{0}{\mathfrak M}

\usepackage{verbatim}

%% file: Canonical_parts_ArXiv.bbl
\begin{thebibliography}{10}

\bibitem{ForcingBlocks}
J.~Carmesin, R.~Diestel, M.~Hamann, and F.~Hundertmark.
\newblock $k$-{B}locks: a connectivity invariant for graphs.
\newblock Preprint 2013.

\bibitem{CDHH13CanonicalAlg}
J.~Carmesin, R.~Diestel, M.~Hamann, and F.~Hundertmark.
\newblock Canonical tree-decompositions of finite graphs {I.~Existence} and
  algorithms.
\newblock {\em J. Combin. Theory Ser. B}, to appear.

\bibitem{confing}
J.~Carmesin, R.~Diestel, F.~Hundertmark, and M.~Stein.
\newblock Connectivity and tree structure in finite graphs.
\newblock {\em Combinatorica}, 34(1):1--35, 2014.

\bibitem{CG14:isolatingblocks}
Johannes Carmesin and Pascal Gollin.
\newblock Canonical tree-decompositions isolating all their $k$-blocks.
\newblock Presentation at \emph{Hamburg workshop on graphs and matroids},
  Spiekeroog~2014.

\bibitem{DiestelBook10noEE}
R.~Diestel.
\newblock {\em {Graph Theory}}.
\newblock Springer, 4th edition, 2010.

\bibitem{profiles}
F.~Hundertmark.
\newblock Profiles. {A}n algebraic approach to combinatorial connectivity.
\newblock arXiv:1110.6207, 2011.

\bibitem{ReedConnectivityMeasure}
B.A. Reed.
\newblock Tree width and tangles: a new connectivity measure and some
  applications.
\newblock In R.A. Bailey, editor, {\em Surveys in Combinatorics}. Cambridge
  Univ.\ Press, 1997.

\bibitem{GMX}
N.~Robertson and P.D. Seymour.
\newblock Graph minors. {X}. {O}bstructions to tree-decomposition.
\newblock {\em J.~Combin.\ Theory (Series B)}, 52:153--190, 1991.

\bibitem{ST1993GraphSearching}
P.~Seymour and R.~Thomas.
\newblock Graph searching and a min-max theorem for tree-width.
\newblock {\em J.~Combin.\ Theory (Series B)}, 58(1):22 -- 33, 1993.

\bibitem{TutteGrTh}
W.~T. Tutte.
\newblock {\em Graph Theory}.
\newblock Addison-Wesley, 1984.

\end{thebibliography}
